\documentclass[11pt,a4paper, final, twoside]{article}
\usepackage{amsmath}
\usepackage{fancyhdr}
\usepackage{amsthm}
\usepackage{amsfonts}
\usepackage{amssymb}
\usepackage{amscd}
\usepackage{amsthm}
\usepackage{graphicx}
\usepackage{afterpage}
\usepackage[colorlinks=true, urlcolor=blue,  linkcolor=blue, citecolor=blue]{hyperref}

\newtheorem{theorem}{Theorem}

\newtheorem{corollary}[theorem]{Corollary}

\newtheorem{definition}[theorem]{Definition}

\newtheorem{lemma}[theorem]{Lemma}

\newtheorem{proposition}[theorem]{Proposition}
\newtheorem{remark}[theorem]{Remark}

\numberwithin{equation}{section}

\begin{document}
\hyphenpenalty=10000

\begin{center}
{\Large \textbf{Applications of Asymptotic Riesz Representation Theorem}}\\[5mm]
{\large {Simona Macovei*  }\\[10mm]
}
\end{center}

{\footnotesize \textbf{Abstract}.
    We review the relation between compact asymptotic spectral measures and certain positive asymptotic morphism on locally compact spaces via asymptotic Riesz representation theorem, as introduced by Martinez and Trout [3]. Applications to this theorem shall be discus.}

\footnote{\textsf{*simonamacovei@yahoo.com}}

\afterpage{
\fancyhead{} \fancyfoot{} 
\fancyhead[LE, RO]{\bf\thepage}
\fancyhead[LO]{\small Applications of Asymptotic Riesz Representation Theorem}
\fancyhead[RE]{\small Simona Macovei  }
}

\section{Introduction}

In [3], Martinez and Trout introduced the positive asymptotic morphism, defined as:

\noindent
\noindent A \textit{positive asymptotic morphism} from a C* - algebra \textit{A} to a C* - algebra \textit{B} is a family of maps ${\left\{Q_h\right\}}_{h\in \left.(0,1\right]}:A\ \to B$, parameterized by $h\in \left.(0,1\right],$ such that the following hold:
\begin{enumerate}
\item Each $Q_h$ is a positive linear map;
\item The map $\left.h\mapsto Q_h\left(f\right):(0,1\right]\to B\left(H\right){\rm is\ continuous\ for\ each\ }f\in A$;
\item For all $f,\ g\in $\textit{ A }we have 
\end{enumerate}

\noindent 
\[{\mathop{\lim }_{h\to 0} \left\|Q_h\left(fg\right)-Q_h(f)Q_h(g)\right\|\ }=0.\]

Also, Martinez and Trout introduced the concept of \textit{an asymptotic spectral measure} ${\left\{A_h\right\}}_{h\in \left.(0,1\right]}:\Sigma \to B(H)$ associated to a measurable space $(X,\Sigma )$ (Definition 3.1).  Roughly, \textit{an asymptotic spectral measure} ${\left\{A_h\right\}}_{h\in \left.(0,1\right]}:\Sigma \to B(H)$ is a continuous family of positive operator-valued measures which has the property:

\[{\mathop{lim}_{h\to 0} \left\|A_h\left({\Delta }_1{\bigcap \Delta }_2\right)-A_h({\Delta }_1)A_h({\Delta }_2)\right\|\ }=0,\]
for each $\Delta $ ${}_{1}$, $\Delta $ ${}_{2}$  $\in $ $\Sigma $.

\noindent 

Let \textit{X}  be a locally compact Hausdorff topological space with Borel C*-algebra ${\Sigma }_{{\rm X}}$ and let ${{\rm C}}_{{\rm X}}\subset {\Sigma }_{{\rm X}}$ denote the collection of all pre-compact open subsets of \textit{X}. Let $C_0(X)\ $ denote the C* - algebra of all continuous functions which vanish at infinity on \textit{X}. Define $B_0(X)\ $to be the C* - subalgebra of $B_b(X)\ $ (C* - algebra of all bounded Borel functions on \textit{X}) generated by $\left\{\left.{\chi }_U\right|U\in C_X\right\}$, where ${\chi }_U$ denotes the characteristic function of $U\subseteq X$.

Let \textit{H} be a separable Hilbert space, $B\left(H\right)$ be the C* - algebra of all bounded linear operators on \textit{H } and B denote a hereditary C* - algebra of $B\left(H\right)$.

The asymptotic Riesz representation theorem, formulated by Martinez and Trout [3], gives a bijective correspondence between positive asymptotic morphisms ${\left\{Q_h\right\}}_{h\in \left.(0,1\right]}:B_0(X)\ \to B(H)$, having property  $Q_h(C_0(X))\ \subset \ B\ $ for any $h\in \left.(0,1\right]$, and asymptotic spectral measures ${\left\{A_h\right\}}_{h\in \left.(0,1\right]}:{ÓÓ}_X\to B(H)$, having property  $A_h(C_X)\ \subset \ B$ for any $h\in \left.(0,1\right]$. This correspondence is given by

\[Q_h(f)=\int_X{f(x)dA_h(x)},\] 
\textit{}

\noindent for any $f\in B_0\left(X\right).$ (Teorema 4.2. [3])

In this paper, we study some applications of this theorem.

\noindent 

\section{Positive Asymptotic Morphisms}

\noindent 

In this section we review the basic definitions and properties of positive asymptotic morphisms. 

Let \textit{A} and \textit{B} be two C* - algebras. An linear operator \textit{Q : A $\to $ B}  is \textit{positive} if $Q(f)\ge 0$ for any $f\ge 0$.

\noindent 

\begin{definition} A \textit{positive asymptotic morphism} from a C* - algebra \textit{A} to a C* - algebra \textit{B} is a family of maps ${\left\{Q_h\right\}}_{h\in \left.(0,1\right]}:A\ \to B$, parameterized by $h\in \left.(0,1\right],$ such that the following hold:

\begin{enumerate}
\item  Each $Q_h$ is a positive linear map;

\item  The map $\left.h\mapsto Q_h\left(f\right):(0,1\right]\to B\left(H\right){\rm is\ continuous\ for\ each\ }f\in A$;

\item  For all $f,\ g\in $\textit{ A }we have 
\end{enumerate}

\[{\mathop{\lim }_{h\to 0} \left\|Q_h\left(fg\right)-Q_h(f)Q_h(g)\right\|\ }=0.\] 

\label{def 2.1}
\end{definition}

\noindent (Definition 2.1. [3])

\begin{definition} Two positive asymptotic morphisms ${\left\{Q_h\right\},\left\{P_h\right\}\ }_{h\in \left.(0,1\right]}:A\ \to B$ are called \textit{asymptotically equivalent} if for all $f\in A$\textit{  }we have that 
\[{\mathop{lim}_{h\to 0} \left\|Q_h\left(f\right)-P_h(f)\right\|\ }=0.\](Definition 2.2. [3])
\label{def 2.2}
\end{definition}
\noindent 

\begin{remark} The asymptotic equivalence relation of two positive asymptotic morphisms is symmetric, reflexive and transitive.
\end{remark} 
\noindent 

Let \textit{H} be a separable Hilbert space and $B\left(H\right)$ be the C* - algebra of all bounded linear operators on \textit{H }. Let \textit{X} be a set equipped with a C*-algebra ${\Sigma }_{{\rm X}}$ of measurable sets and let ${{\rm C}}_{{\rm X}}\subset {\Sigma }_{{\rm X}}$ denote the collection of all pre-compact open subsets of \textit{X}. Define $B_0(X)\ $to be the C* - subalgebra of $B_b(X)\ $ (C* - algebra of all bounded Borel functions on \textit{X}) generated by $\left\{\left.{\chi }_U\right|U\in C_X\right\}$, where ${\chi }_U$ denotes the characteristic function of $U\subseteq X$.  If \textit{X} is also $\sigma $ -- compact, then let $C_0(X)\ $ be the set of all continuous functions which vanish at infinity on \textit{X}.

\noindent 

\noindent We call \textit{support} of a morphism \textit{Q : }$C_0(X)$\textit{$\to $}$\ B\left(H\right)$ the set

\[\ supp\left(Q\right)=\cap \left\{\left.F\subset X\right|F\ closed\ and\ Q\left(f\right)=0,\ \forall f\ cu\ supp(f)\subset X\backslash F\right\}.\]

\noindent A morphism $Q : C_0(X)\ \to B\left(H\right)$ will be said to have\textit{ compact support }if there is a compact subset \textit{K }of\textit{ X }such that\textit{  }$supp(Q)$\textit{ $\subset $ K.}

\noindent 

\begin{definition} A positive asymptotic morphism ${\left\{Q_h\right\}}_{h\in \left.(0,1\right]}:B_b(X)\ \to B\left(H\right)$ will be said to have \textit{compact support} if there is a compact subset \textit{K }of\textit{ X }such that  $supp(Q_h)$\textit{ $\subset $ K, }$\forall h\in \left.(0,1\right]$.
\label{def 2.3}
\end{definition}

\noindent 

\begin{definition} Let ${\left\{Q_h\right\}}_{h\in \left.(0,1\right]}:B_b(X)\ \to B\left(H\right)$ be a positive asymptotic morphism. The \textit{support of } $\left\{Q_h\right\}$ is defined as the set

\[supp\left(\left\{Q_h\right\}\right)=\cap \left\{\left.F\ closed\ \right|{\mathop{lim}_{h\to 0} \left\|Q_h(f)\right\|\ }=0,\ \forall f\in B_b\left(X\right)\ with\ supp\left(f\right)\bigcap F=\emptyset \ \right\}.\] 

\label{def 2.4}
\end{definition}

\begin{remark}
\begin{enumerate}
\item $supp\left(\left\{Q_h\right\}\right)\subseteq \bigcup_{h\in (\left.0,1\right]}{supp\left(Q_h\right)}$.

\item If $\left\{Q_h\right\}$ has compact support $supp\left(\left\{Q_h\right\}\right)$ is a compact set.
\end{enumerate}
\end{remark}

\begin{theorem} 
Let ${\left\{Q_h\right\}}_{h\in \left.(0,1\right]}:B_0(X)\ \to B\left(H\right)$  be a positive asymptotic morphism such that ${\mathop{lim}_{h\to 0} \left\|Q_h\left(1\right)-I\right\|\ }=0$\textit{. }Then 

\[Sp\left(\left\{Q_h\left(f\right)\right\}\right)\subseteq f\left(supp\left(\left\{Q_h\right\}\right)\right),\] 
for any $f\in B_0(X)$.
\label{theorem 2.1}
\end{theorem}

\begin{proof} Let $f\in B_0(X)$ and $\lambda \notin f(supp\left(\left\{Q_h\right\}\right)$. Then $\lambda -f(x)\ne 0$, $\forall x\in supp\left(\left\{Q_h\right\}\right)$. Thus there is a an open set $G\supset supp\left(\left\{Q_h\right\}\right)$ such that $\lambda -f(x)\ne 0$, $\forall x\in G$. Therefore, the application $x\longmapsto 1/(f\left(x\right)-\lambda )\in C(G)$. 

\noindent Let $g\in B_0(X)$   such that  $g\left(x\right)=1/(f\left(x\right)-\lambda )$, $\forall x\in G$. Then

\[g\left(x\right)\left(f\left(x\right)-\lambda \right)=\left(f\left(x\right)-\lambda \right)g\left(x\right)=1,\]

\noindent $\forall x\in G$. Taking into account the above relation and since ${\mathop{\lim }_{h\to 0} \left\|Q_h\left(1\right)-I\right\|\ }=0$, it follows that

\[{\mathop{\lim }_{h\to 0} \left\|Q_h\left(g\right)\left(\lambda -Q_h\left(f\right)\right)-I\right\|\ }=\]
\[{\mathop{\lim }_{h\to 0} \left\|Q_h\left(g\right)\left(\lambda -Q_h\left(f\right)\right)-Q_h\left(g\right)Q_h\left(\lambda -f\right)+Q_h\left(g\right)Q_h\left(\lambda -f\right)-Q_h\left(g(\lambda -f)\right)+\right .} \]
\[{\left . Q_h\left(g(\lambda -f)\right)-Q_h\left(1\right)+Q_h\left(1\right)-I\right\|\ }\le\]
\[\le {\mathop{\lim }_{h\to 0} \left\|Q_h\left(g\right)\left(\lambda -Q_h\left(f\right)\right)-Q_h\left(g\right)Q_h\left(\lambda -f\right)\right\|\ }+\]
\[+{\mathop{\lim }_{h\to 0} \left\|Q_h\left(g\right)Q_h\left(\lambda -f\right)-Q_h\left(g(\lambda -f)\right)\right\|\ }+\]
\[+{\mathop{\lim }_{h\to 0} \left\|Q_h\left(g(\lambda -f)\right)-Q_h\left(1\right)\right\|\ }+{\mathop{\lim }_{h\to 0} \left\|Q_h\left(1\right)-I\right\|\ }=0.\]

\noindent Analogously ${\mathop{\lim }_{h\to 0} \left\|\left(\lambda -Q_h\left(f\right)\right)Q_h\left(g\right)-I\right\|\ }=0$. Therefore  $\lambda \in r(\left\{Q_h\left(f\right)\right\})$.

\noindent We have showed that  
\[Sp\left(\left\{Q_h\left(f\right)\right\}\right)\subseteq f\left(supp\left(\left\{Q_h\right\}\right)\right), \forall f\ \in B_0(X).\] 
\end{proof}

\begin{corollary} Let ${\left\{Q_h\right\}}_{h\in \left.(0,1\right]}:B_0{\rm (}{\mathbb C}{\rm )}{\rm \ }\to B\left(H\right)$ be a positive asymptotic morphism such that ${\mathop{lim}_{h\to 0} \left\|Q_h\left(1\right)-I\right\|\ }=0$. Then

\[Sp\left(\left\{Q_h\left(z\right)\right\}\right)\subseteq supp\left(\left\{Q_h\right\}\right).\]
 
\end{corollary}

\begin {proof} We take in Theorem \ref{theorem 2.1} $f=z$, where \textit{z} represents the identity application.

\end{proof} 

\section{Asymptotic Spectral Measures}

\noindent 

Let (\textit{X}, $\Sigma $) be a measurable space and \textit{H} a separable Hilbert space. Let $\varepsilon \subset \Sigma $ be a fix collection of measurable sets.

\noindent 

\begin{definition} A \textit{positive operator-valued measure} on the measurable space (X, $\Sigma $) is a mapping $A:\Sigma \to B(H)$ which satisfies the following properties:

\begin{enumerate}
\item  $A\left(\emptyset \right)=0$;

\item  $A\left(\Delta \right)\ge 0$, $\forall \Delta \in \Sigma$;

\item  $A\left(\bigcup^{\infty }_{n=1}{{\Delta }_n}\right)=\sum^{\infty }_{n=1}{A({\Delta }_n)}$, for disjoint measurable sets ${({\Delta }_n)}^{\infty }_{n=1}\subset \Sigma $, where the series converges in weak operator topology.
\end{enumerate}
\label{def 3.1}
\end{definition}
\noindent 

\begin{definition} An \textit{asymptotic spectral measure} on (X, $\Sigma $, $\varepsilon $) is a family of maps ${\left\{A_h\right\}}_{h\in \left.(0,1\right]}:\Sigma \to B(H)$\textit{, }parameterized by $h\in \left.(0,1\right],$ such that the following hold:

\noindent i) Each $A_h$ is a positive operator-valued measure;

\noindent ii) ${\mathop{lim}_{h\to 0} \left\|A_h(X)\right\|\le 1\ }$;

\noindent iii) The map $\left.h\mapsto A_h\left(\Delta \right):(0,1\right]\to B(H)$ is continuous, for any $\Delta \in \varepsilon$;

\noindent iv) For each $\Delta_{1}$, $\Delta_{2}\in \varepsilon$ we have

\[{\mathop{lim}_{h\to 0} \left\|A_h\left({\Delta }_1{\bigcap \Delta }_2\right)-A_h({\Delta }_1)A_h({\Delta }_2)\right\|\ }=0.\]

\noindent The triple (\textit{X}, $\Sigma $, \textit{$\varepsilon $}) will be called \textit{asymptotic measure space}. 

\noindent If \textit{$\varepsilon $} = $\Sigma $, then $\left\{A_h\right\}$ will be called \textit{total (full) asymptotic spectral measure} on (\textit{X}, $\Sigma $).

\noindent If each $A_h$ is normalized, i.e.  $A_h(X{\rm )}$ = \textit{I${}_{H}$}, then $\left\{A_h\right\}$ will be called \textit{normalized.} (Definition 3.1. [3])

\noindent If ${\mathop{\lim }_{h\to 0} \left\|A_h(X)-\ I_H\right\|\ }=0$, then $\left\{A_h\right\}$ will be called \textit{asymptotically normalized.} 
\label{def 3.2}
\end{definition}
\noindent 

\begin{definition} Two asymptotic spectral measures ${\left\{A_h\right\},\ \left\{B_h\right\}}_{h\in \left.(0,1\right]}:\Sigma \to B(H)$ on (X, $\Sigma $) are said to be \textit{(asymptotically) equivalent} if for each measurable set  $\Delta \in \varepsilon$, we have

\[{\mathop{lim}_{h\to 0} \left\|A_h\left(\Delta \right)-B_h(\Delta )\right\|\ }=0.\] 

\label{def 3.3}
\end{definition}
\noindent (Definition 2.2 [3])

\noindent Let \textit{X} denote a locally compact Hausedorff topological space with Borel $\sigma-algebra \Sigma$.

\begin{definition} Let\textit{ A }be a Borel positive operator-valued measure on\textit{ X. }The \textit{cospectrum }of\textit{ A }is defined as the set

\[cospec(A)\ =\ \bigcup \left\{\left.U\subset X\right|U\ {\rm is\ open\ and}\ A\left(U\right)=0\right\}.\]

\noindent The \textit{spectrum}  of \textit{A }is the complement, i.e.

\[spec{(A) = X \backslash  } cospec(A).\]

\label{def 3.4}
\end{definition}

\begin{definition} Let\textit{ A }be a Borel positive operator-valued measure on\textit{ X. A }is said to be \textit{compact} if $spec\left(A\right)$ is a compact subset of \textit{X.}

\label{def 3.5}
\end{definition}

\begin{theorem} Let\textit{ A }be a  Borel positive operator-valued measure on\textit{ X. }Then

\[A\left(spec\left(A\right)\right)=A\left(X\right).\] 

\label{theorem 3.6}
\end{theorem}

\noindent (Theorem 23 [5])

\begin{definition} An asymptotic spectral measure $\left\{A_h\right\}$ on \textit{X} will have\textit{ compact support} if there is a compact subset \textit{K  }of\textit{ X  }such that $spec(A_h)\subset \ K$\textit{, }$\forall h\in \left.(0,1\right]$. 
\noindent(Definition 3.4 [3])

\label{def 3.7}
\end{definition}

\begin{remark} i) If$\ \left\{A_h\right\}$ has compact, then $A_h$ has compact support, $\forall h\in \left.(0,1\right]$;

\noindent ii) If $\left\{A_h\right\}$ has compact support, then

\[A_h(spec(A_h))=A_h(K)=A_h(X), \forall h\in \left.(0,1\right].\] 

\end{remark}

\begin{definition} Let  ${\left\{A_h\right\}}_{h\in \left.(0,1\right]}:{\Sigma }_X\to B(H)$ be an asymptotic spectral measure. The \textit{cospectrum}  of $\left\{A_h\right\}$ is defined as the set

\[cospec(\left\{A_h\right\})=\bigcup \left\{\left.a\subset X\right|a\ open\ and\ {\mathop{lim}_{h\to 0} \left\|A_h(a)\right\|\ }=0\ \right\}.\]

\noindent The \textit{spectrum} of $\left\{A_h\right\}$ is the complement of $cospec(\left\{A_h\right\}{\rm )}$, i.e.

\[spec\left(\left\{A_h\right\}\right)=X\backslash cospec(\left\{A_h\right\}).\] 

\label{def 3.8}
\end{definition}

\begin{remark} 

\noindent i) $spec\left(\left\{A_h\right\}\right)\subseteq \bigcup_{h\in (\left.0,1\right]}{spec\left(A_h\right)}$ and $\bigcap_{h\in (\left.0,1\right]}$ ${cospec(A_h)}\subseteq cospec(\left\{A_h\right\}$.

\noindent ii) If  $\left\{A_h\right\}$ has compact support, then $spec\left(\left\{A_h\right\}\right)$ is also a compact set.

\label{rem 3.9}
\end{remark} 

\begin{proof} i) Let $a\subset \bigcap_{h\in (\left.0,1\right]}{cospec(A_h{\rm )}}$ be an open set. Thus $A_h\left(a\right)=0$, $\forall $$h\in (\left.0,1\right]$ and

\[{\mathop{\lim }_{h\to 0} \left\|A_h(a)\right\|\ }{\rm =0}.\] 

\noindent Therefore

\[a\subset \ cospec(\left\{A_h\right\}{\rm )}, \forall a\subset \bigcap_{h\in (\left.0,1\right]}{cospec(A_h{\rm )}}.\] 

\noindent It results

\[\bigcap_{h\in (\left.0,1\right]}{cospec(A_h{\rm )}}\subseteq cospec(\left\{A_h\right\}{\rm )}\]

\noindent and, taking the complement we have

\[spec\left(\left\{A_h\right\}\right)\subseteq \bigcup_{h\in (\left.0,1\right]}{spec\left(A_h\right)}.\]

\noindent ii) Let $\left\{A_h\right\}$ be a compact set. Thus there is a compact subset \textit{K} of \textit{X} such that $spec\left(A_h\right)\subset K$, $\forall $$h\in (\left.0,1\right]$. 

\noindent By i), it follows

\[spec\left(\left\{A_h\right\}\right)\subseteq \bigcup_{h\in (\left.0,1\right]}{spec\left(A_h\right)}\subset K.\] 

\end{proof}

\begin{lemma} Let ${\left\{A_h\right\}}_{h\in \left.(0,1\right]}:{\Sigma }_X\to B(H)$ be an asymptotic spectral measure. Then

\[{\mathop{lim}_{h\to 0} \left\|A_h(K)\right\|\ }=0,\]

\noindent for each compact subset \textit{K }of $cospec(\left\{A_h\right\})$.

\label{lemma 3.10}
\end{lemma} 

\begin{proof} Let \textit{K} be a compact subset of $cospec{\rm (}\left\{A_h\right\}{\rm ).}$ Thus each element of \textit{K} belongs to an open set  $a$ having property ${\mathop{\lim }_{h\to 0} \left\|A_h(a)\right\|\ }{\rm =0}$. Since \textit{K} is a compact set, hence there is a family of open set ${(a_i)}^n_1\subset X$ such that $K\subset a_1\bigcup{\dots \bigcup a_n}$. Therefore

\[A\left(K\right)\le A\left(a_1\right)+\dots +A\left(a_n\right)=0\] 
 and
\[{\mathop{\lim }_{h\to 0} \left\|A_h(K)\right\|\ }\le {\mathop{\lim }_{h\to 0} \left\|A_h(a_1)\right\|\ }+\dots +{\mathop{\lim }_{h\to 0} \left\|A_h(a_n)\right\|\ }=0\] 

\end{proof}

\begin{proposition} Let ${\left\{A_h\right\}}_{h\in \left.(0,1\right]}:{\Sigma }_X\to B(H)$ be an asymptotic spectral measure. Then

\[\overline{\mathop{lim}}_{h\to 0} \left\|A_h(X)\right\|\ =\overline{\mathop{lim}}_{h\to 0} \left\|A_h(spec\left(\left\{A_h\right\}\right))\right\|\ .\] 

\label{prop 3.11}
\end{proposition}

\begin{proof} We show that

\[{\mathop{\lim }_{h\to 0} \left\|A_h(cospec\left(\left\{A_h\right\}\right))\right\|\ }=0.\]

\noindent Let \textit{K} be a compact subset of $cospec{\rm (}\left\{A_h\right\}{\rm ).}$ By Lemma \ref{lemma 3.10}, it follows that

\[{\mathop{\lim }_{h\to 0} \left\|A_h(K)\right\|\ }{\rm =0}.\]

\noindent Since $A_h$ is regular, for any $h\in \left.(0,1\right]$, by above relation, it results that

\[{\mathop{\lim }_{h\to 0} \left\|A_h(cospec\left(\left\{A_h\right\}\right))\right\|\ }=0.\] 

Since

\[spec\left(\left\{A_h\right\}\right){\rm =X}{\rm \backslash }cospec\left(\left\{A_h\right\}\right){\rm ,\ }\] 

we have that

\noindent 
\[{\overline{\mathop{\lim }}_{h\to 0} \left\|A_h(X)\right\|\ }={\overline{\mathop{\lim }}_{h\to 0} \left\|A_h\left(spec\left(\left\{A_h\right\}\right)\right)+A_h\left(cospec\left(\left\{A_h\right\}\right)\right)\right\|\ }\le\]
\[\le {\overline{\mathop{\lim }}_{h\to 0} \left\|A_h(spec\left(\left\{A_h\right\}\right))\right\|\ }+{\overline{\mathop{\lim }}_{h\to 0} \left\|A_h(cospec\left(\left\{A_h\right\}\right))\right\|\ }=\]
\[={\overline{\mathop{\lim }}_{h\to 0} \left\|A_h(spec\left(\left\{A_h\right\}\right))\right\|\ }.\]

\noindent In addition, we have that

\[\overline{\mathop{\lim }}_{h\to 0} \left\|A_h(spec\left(\left\{A_h\right\}\right))\right\|\ =\]
\[=\overline{\mathop{\lim }}_{h\to 0} \left\|A_h\left(spec\left(\left\{A_h\right\}\right)\right)+A_h\left(cospec\left(\left\{A_h\right\}\right)\right)-A_h(cospec\left(\left\{A_h\right\}\right))\right\|\ \le\]
\[\le {\overline{\mathop{\lim }}_{h\to 0} \left\|A_h(cospec\left(\left\{A_h\right\}\right))\right\|\ }+{\overline{\mathop{\lim }}_{h\to 0} \left\|A_h\left(cospec\left(\left\{A_h\right\}\right)\right)+A_h(spec\left(\left\{A_h\right\}\right))\right\|\ }\le\]
\[\le {\overline{\mathop{\lim }}_{h\to 0} \left\|A_h(X)\right\|\ }.\]

\noindent From two preceding relations, it follows that

\[{\overline{\mathop{\lim }}_{h\to 0} \left\|A_h(X)\right\|\ }={\overline{\mathop{\lim }}_{h\to 0} \left\|A_h(spec\left(\left\{A_h\right\}\right))\right\|\ }.\] 

\end{proof}

\begin{theorem} Let  ${\left\{A_h\right\},\ \left\{B_h\right\}}_{h\in \left.(0,1\right]}:{\Sigma }_X\to B(H)$ be two asymptotic spectral measures. If $\left\{A_h\right\}$, $\ \left\{B_h\right\}$ are asymptotically equivalent, then

\[spec\left(\left\{A_h\right\}\right)=spec\left(\left\{B_h\right\}\right).\] 

\label{theorem 3.12}
\end{theorem}

\begin{proof} Let be an open set  $a\subset \ cospec\left(\left\{A_h\right\}\right)$. Thus

\[{\mathop{\lim }_{h\to 0} \left\|A_h(a)\right\|\ }{\rm =0.}\]

\noindent Since $\left\{A_h\right\},\ \left\{B_h\right\}$ are asymptotically equivalent, it results that

\[{\mathop{\lim }_{h\to 0} \left\|A_h\left(a\right)-B_h\left(a\right)\right\|\ }{\rm =0}.\]

\noindent By two preceding relations, we have that

\[{\mathop{\lim }_{h\to 0} \left\|B_h\left(a\right)\right\|\ }={\mathop{\lim }_{h\to 0} \left\|A_h\left(a\right)\right\|\ }{\rm =0}.\]

\noindent Thus $a\subset \ cospec\left(\left\{B_h\right\}\right)$, $\forall $$a\subset cospec\left(\left\{A_h\right\}\right)$ open. Therefore

\[spec\left(\left\{B_h\right\}\right)\subset spec\left(\left\{A_h\right\}\right).\] 

\noindent Reciprocal: Analog.

\end{proof} 

\begin{remark} Let $\left\{A_h\right\}$, $\ \left\{B_h\right\}$ be two asymptotic spectral measures on (X, $\Sigma $). If $\left\{A_h\right\}$, $\ \left\{B_h\right\}$ are asymptotically equivalent, then $\left\{A_h\right\}$ has compact support if and only if $\left\{B_h\right\}$ has compact support.

\end{remark} 

\begin{proof} By preceding Proposition, for each compact subset\textit{ K} of \textit{X}, we have $spec\left(\left\{A_h\right\}\right)\subset K$ if and only if  $spec\left(\left\{B_h\right\}\right)\subset K$.

\end{proof} 

\begin{proposition} Let $\left\{A_h\right\}$ be a full asymptotic spectral measure on (\textit{X, B}), where \textit{B} is the $\sigma $ -- algebra of Borel subsets of \textit{X}, and $a\in {\rm B}$. Then $\left\{A^a_h\right\}:B\to B(H)$, parameterized by $h\in \left.(0,1\right]$, given by

\[A^a_h\left(b\right)=A_h(a\cap b), \forall b\in B and\ \forall h\in \left.(0,1\right],\]  

\noindent is an  asymptotic spectral measure.

\label{prop 3.13}
\end{proposition} 

\begin{proof} By definition of $\left\{A^a_h\right\}$, we have

\[A^a_h\left(\emptyset \right)=A_h\left(a\cap \emptyset \right)=A_h\left(\emptyset \right)=0,  \forall h\in \left.(0,1\right]\] 

and

\[{\mathop{\overline{\lim }}_{h\to 0} \left\|A^a_h(X)\right\|\ }={\mathop{\overline{\lim }}_{h\to 0} \left\|A_h(a\cap X)\right\|\ }=\]
\[={\mathop{\overline{\lim }}_{h\to 0} \left\|A_h(a)\right\|\ }\le {\mathop{\overline{\lim }}_{h\to 0} \left\|A_h(X)\right\|\ }\le 1.\]

\noindent Let ${(b_n)}_{n\in {\mathbb N}}\subset B$ be a family of disjoint sets. Thus ${(a\cap b_n)}_{n\in {\mathbb N}}\subset B$ is also a family of disjoint sets. Since $A_h\ $is numerable additive, $\forall $$h\in \left.(0,1\right]$, it results

\[A^a_h\left(\bigcup_{n\in {\mathbb N}}{b_n}\right)=A_h\left(\bigcup_{n\in {\mathbb N}}{\left(a\cap b_n\right)}\right)=\] 
\[=\sum_{n\in {\mathbb N}}{A_h(a\cap b_n)}=\sum_{n\in {\mathbb N}}{A^a_h(b_n)}, \forall h\in \left.(0,1\right].\]

\noindent As the map $\left.(0,1\right]\to B(H):h\to A_h(a\cap b)$ is continuous $\forall \ b\in B$, then the map is also continuous $\forall b\in B$.

\noindent Let $b_1,b_2\in B$. Thus

\[{\mathop{\overline{\lim }}_{h\to 0} \left\|A^a_h\left(b_1{\bigcap b}_2\right)-A^a_h(b_1)A^a_h(b_2)\right\|\ }=\]
\[={\mathop{\overline{\lim }}_{h\to 0} \left\|A_h\left(a\cap b_1{\bigcap b}_2\right)-A_h(a\cap b_1)A_h({a\cap b}_2)\right\|\ }=\]
\[={\mathop{\lim }_{h\to 0} \left\|A_h\left((a\cap b_1){\bigcap (a\cap b}_2\right))-A_h(a\cap b_1)A_h({a\cap b}_2)\right\|\ }=0.\]

\noindent Therefore, $\left\{A^a_h\right\}:B\to B(H)$ is a full asymptotic spectral measure.

\end{proof} 

\begin{proposition} Let $\left\{A_h\right\}$ be a full asymptotic spectral measure on (\textit{X, B}) and $\left\{A^a_h\right\}:B\to B(H{\rm )}$, parameterized by $h\in \left.(0,1\right]$, given by $A^a_h\left(b\right)=A_h(a\cap b)$, $\forall b\in B$ and $\forall h\in \left.(0,1\right]$. Then

\[spec(A^a_h)\subseteq \overline{a}\cap spec(A_h), \forall h\in \left.(0,1\right].\] 

\label{prop 3.14}
\end{proposition}

\begin{proof} Let \textit{b} be a compact set such that $b\subset {\mathbb C}\backslash \overline{a}$. Thus $a\cap b=\emptyset $. By this relation we have

\[A^a_h\left(b\right)=A_h\left(a\cap b\right)=A_h\left(\emptyset \right)=0, \forall h\in \left.(0,1\right],\ \] 

\noindent hence

\[b\subset cospec(A^a_h{\rm )}\Rightarrow {\mathbb C}\backslash \overline{a}\subset cospec(A^a_h), \forall h\in \left.(0,1\right]\]

\noindent (by regularity property of measures $A_h$ - Teorema 23 $\left[5\right]$). Therefore

\[spec(A^a_h)\subseteq \overline{a}, \forall h\in \left.(0,1\right].\]

\noindent Let \textit{b} be a compact set such that $b\subset {\mathbb C}\backslash spec(A_h)$. Thus there is a family of open sets ${(b_i)}_{i=\overline{1,n}\ }$ such that

\[b\subset \bigcup^n_{i=1}{b_i,\ \ {b_i\subset {\mathbb C}\backslash spec(A_h)\Rightarrow A}_h\left(b_i\right)=0}.\]

\noindent Since each $A_h$ is additive, we have

\[A_h\left(b\right)\le A_h\left(\bigcup^n_{i=1}{b_n}\right)=\sum^n_{i=1}{A_h\left(b_i\right)}=0.\]

\noindent Taking into account the following relation

\[A^a_h\left(b\right)=A_h\left(a\cap b\right)\le \sum^n_{i=1}{A_h\left(a\cap b_i\right)}\le \sum^n_{i=1}{A_h\left(b_i\right)}=0\]
 
\noindent it results

\[b\subset {\mathbb C}\backslash spec\left(A^a_h\right),\]

\noindent for any compact set \textit{b} such that $b\subset {\mathbb C}\backslash spec(A_h)$. Since each $A_h$ is regular, it follows

\[{\mathbb C}\backslash spec(A_h{\rm )}\subseteq {\mathbb C}\backslash spec(A_h), \forall h\in \left.(0,1\right].\] 

\noindent Therefore

\[spec(A^a_h)\subseteq spec(A_h), \forall h\in \left.(0,1\right].\] 

\end{proof}

\begin{remark} If $\left\{A_h\right\}$ is an asymptotic spectral measure having compact support, then  $\left\{A^a_h\right\}$ is an asymptotic spectral measure having compact support, $\forall a\in B$.
\end{remark} 

\begin{proposition} Two full asymptotic spectral measures on (\textit{X, B}) $\left\{A_h\right\}$, $\left\{B_h\right\}$ are asymptotically equivalent if and only if $\left\{A^a_h\right\},\ \left\{B^a_h\right\}:B\to B(H)$, given by $A^a_h\left(b\right)=A_h(a\cap b)$ and $B^a_h\left(b\right)=B_h(a\cap b)$, $\forall b\in B$, $\forall h\in \left.{\rm (0,1}\right]$, are asymptotically equivalent $\forall a\in B$.

\label{prop 3.15}
\end{proposition}

\begin{proof} Let $a\in B$ be fixed. Since $\left\{A_h\right\}$, $\left\{B_h\right\}$ are asymptotically equivalent, thus

\[{\mathop{\lim }_{h\to 0} \left\|A_h\left({\rm a}\bigcap {\rm b}\right)-B_h({\rm a}\bigcap {\rm b})\right\|\ }=0, \forall b\in B. \] 

\noindent It follows that

\[{\mathop{\lim }_{h\to 0} \left\|A^a_h\left({\rm b}\right)-B^a_h\left({\rm b}\right)\right\|\ }=0, \forall b\in B.\]

\noindent Reciprocal. Since $\left\{A^a_h\right\},\ \left\{B^a_h\right\}$ are asymptotically equivalent $\forall $$a\in B\ $ and

\[A^a_h\left(a\right)=A_h\left(a\right), B^a_h\left(a\right)=B_h\left(a\right),\]

\noindent it results 

\[{\mathop{\lim }_{h\to 0} \left\|A_h\left({\rm a}\right)-B_h\left({\rm a}\right)\right\|\ }=\] 
\[={\mathop{\lim }_{h\to 0} \left\|A^a_h\left(a\right)-B^a_h\left(a\right)\right\|\ }=0,\ \forall a\in B. \]

\noindent Therefore, $\left\{A_h\right\}$, $\left\{B_h\right\}$ are asymptotically equivalent.

\end{proof} 

\section{Asymptotic Riesz Representation Theorem}

\noindent 

Let \textit{X}  be a locally compact Hausdorff topological space with Borel $\sigma$-algebra ${\Sigma }_{{\rm X}}$ and let ${{\rm C}}_{{\rm X}}\subset {\Sigma }_{{\rm X}}$ denote the collection of all pre-compact open subsets of \textit{X}. 

Let \textit{H} be a separable Hilbert space, $B\left(H\right)$ be the C* - algebra of all bounded linear operators on \textit{H } and \textit{B} denote a hereditary C* - algebra of $B\left(H\right)$.

\noindent 

\begin{lemma} There is a bijective correspondence between Borel positive operator- valued measure\textit{ A: $\Sigma $${}_{X}$ $\to $ B(H), }having property \textit{A(C${}_{X}$) $\subset $ \textit{B}, }and positive morphism\textit{ Q : C${}_{0}$(X) $\to $B. }This correspondence is given by

\[Q(f)=\ \int_X{f(x)dA(x)}.\]

\label{lemma 4.1}
\end{lemma}
 
 \noindent (Lemma 4.1. [3])
 
 Let $C_0(X)\ $ denote the C* - algebra of all continuous functions which vanish at infinity on \textit{X}. Define $B_0(X)\ $to be the C* - subalgebra of $B_b(X)\ $ (C* - algebra of all bounded Borel functions on \textit{X}) generated by $\left\{\left.{\chi }_U\right|U\in C_X\right\}$, where ${\chi }_U$ denotes the characteristic function of $U\subseteq X$.

\noindent 

\begin{proposition} If ${\left\{A_h\right\}}_{h\in \left.(0,1\right]}:{\Sigma }_X\to B(H)$ is a compact asymptotic spectral measure, then $\left\{A_h\right\}$ verifies property $A_h(C_X)\ \subset \ B$\textit{, }$\forall h\in \left.(0,1\right]$, where \textit{B} is the hereditary subalgebra generated by ${\left\{A_h\left(spec\left(A_h\right)\right)\right\}}_{h\in \left.(0,1\right]}$. 

\label{prop 4.2}
\end{proposition} 

\begin{proof} By Theorem \ref{theorem 3.6} we have

\[A_h\left(cospec\left(A_h\right)\right)=0.\] 

\noindent Let $a\in C_X$. Then

\[0\le A_h\left(a\cap cospec\left(A_h\right)\right)\le A_h\left(cospec\left(A_h\right)\right)=0\] 

\noindent and thus

\[0\le A_h\left(a\cap spec\left(A_h\right)\right)\le A_h\left(spec\left(A_h\right)\right).\]

\noindent Since \textit{B }is the hereditary subalgebra generated by ${\left\{A_h\left(spec\left(A_h\right)\right)\right\}}_{h\in \left.(0,1\right]}$, then $A_h\left(a\right)\in B$.

\end{proof} 

\begin{theorem} (Asymptotic Riesz Representation Theorem): There is a bijectiv correspondence between positive asymptotic morphisms ${\left\{Q_h\right\}}_{h\in \left.(0,1\right]}:B_0(X)\ \to B(H)$, having property $Q_h\left(C_0\left(X\right)\right)\subset \ B$, $\forall h\in \left.(0,1\right]$, and asymptotic spectral measures ${\left\{A_h\right\}}_{h\in \left.(0,1\right]}:{\Sigma }_X\to B(H)$, having property $A_h(C_X)\ \subset \ B$, $\forall h\in \left.(0,1\right]$, given by

\[Q_h\left(f\right)=\int_X{f\left(x\right)dA_h\left(x\right)},\ \forall f\in B_0\left(X\right).\]

\label{theorem 4.3}
\end{theorem}
\noindent(Theorem 4.2. [3])

\section{Application of Asymptotic Riesz Representation Theorem}

\noindent 

\begin{proposition} Let ${\left\{A_h\right\}}_{h\in \left.(0,1\right]}:{\Sigma }_X\to B(H)$ be an asymptotic spectral measure, as in asymptotic Riesz representation theorem, and ${\left\{Q_h\right\}}_{h\in \left.(0,1\right]}:B_0(X)\ \to B(H)$ the corresponding positive asymptotic morphism. Then the following assertions hold:

\begin{enumerate}
\item  $\left\{Q_h\right\}{\rm \ }$is unitary if and only if $\left\{A_h\right\}\ $ is normalized;

\item  ${\mathop{lim}_{h\to 0} \left\|A_h(X)-\ I_H\right\|\ }=0$ if and only if ${\mathop{lim}_{h\to 0} \left\|Q_h(1)-\ I_H\right\|\ }=0$;

\item  Let $\left\{T_h\right\}\subset B(H)$. Then ${\mathop{lim}_{h\to 0} \left\|T_hQ_h\left(f\right)-Q_h\left(f\right)T_h\right\|\ }=0$, $\forall f\in B_0\left(X\right)$ if and only if ${\mathop{lim}_{h\to 0} \left\|T_hA_h\left(\Delta \right)-A_h\left(\Delta \right)T_h\right\|\ }=0$, $\forall \Delta \in {\Sigma }_X$.
\end{enumerate}

\label{prop 5.1}
\end{proposition} 

\begin{proof}  i)  ${\left\{Q_h\right\}}_{h\in \left.(0,1\right]}:B_0(X)\ \to B(H)$ is unitary if $Q_h(1)=I_H\ $, $\forall h\in \left.(0,1\right]$. Since

\[A_h\left(X\right)=Q_h\left({\chi }_X\right)=\int_X{{\chi }_X(x)dA_h(x)}=\int_X{dA_h(x)}=Q_h(1) =\ I_H, \forall h\in \left.(0,1\right], \]

\noindent it follows that ${\left\{A_h\right\}}_{h\in \left.(0,1\right]}$ is normalized.

\noindent Reciprocal. Since ${\left\{A_h\right\}}_{h\in \left.(0,1\right]}:{\Sigma }_{{\rm X}}\to B(H)$ is normalized, i.e. $A_h\left(X\right)=I_H$, $\forall h\in \left.(0,1\right]$, then taking $f=1$ we have

\[Q_h\left(1\right)=\int_X{dA_h(x)}=A_h\left(X\right)=I_H, \forall h\in \left.(0,1\right].\]

\noindent ii) It results from $A_h\left(X\right){=Q}_h(1)$, $\forall h\in \left.(0,1\right]$.

\noindent iii) Since
\[{\mathop{\lim }_{h\to 0} \left\|T_hQ_h\left(f\right)-Q_h\left(f\right)T_h\right\|\ }=0,\ \forall f\in S_0\left(X\right)\subset B_0\left(X\right), \]

\noindent taking $f={\chi }_{\Delta }$ it follows

\[{\mathop{\lim }_{h\to 0} \left\|T_hA_h\left(\Delta \right)-A_h\left(\Delta \right)T_h\right\|\ }=0,\ \forall \Delta \in {\Sigma }_{{\rm X}}.\]
 
\noindent Reciprocal. Since

\[{\mathop{\lim }_{h\to 0} \left\|T_hA_h\left(\Delta \right)-A_h\left(\Delta \right)T_h\right\|\ }=0,\ \forall \Delta \in {\Sigma }_{{\rm X}}, \]

\noindent and having in view $A_h\left(\triangle \right)=Q_h({\chi }_{\Delta })$ it results

\[{\mathop{\lim }_{h\to 0} \left\|T_hQ_h\left({\chi }_{\Delta }\right)-Q_h\left({\chi }_{\Delta }\right)T_h\right\|\ }=0,\ \forall \Delta \in {\Sigma }_{{\rm X}}.\]

\noindent Let $f\in S_0(X)$. Then there are disjoint sets ($\Delta $\textit{${}_{i}$})\textit{${}_{i=1,n}$} such that $f=\sum^n_{i=1}{{\alpha }_i{\chi }_{{\Delta }_i}}$. By above relation, we have

\[{\mathop{\overline{\lim }}_{h\to 0} \left\|T_hQ_h\left(f\right)-Q_h\left(f\right)T_h\right\|\ }={\mathop{\overline{\lim }}_{h\to 0} \left\|\sum^n_{i=1}{{\alpha }_i\left(T_hQ_h\left({\chi }_{{\Delta }_{{\rm i}}}\right)-Q_h\left({\chi }_{{\Delta }_{{\rm i}}}\right)T_h\right)}\right\|\ }\le\]
\[\le \sum^n_{i=1}{{\mathop{\overline{\lim }}_{h\to 0} \left|{\alpha }_i\right|\left\|T_hQ_h\left({\chi }_{{\Delta }_{{\rm i}}}\right)-Q_h\left({\chi }_{{\Delta }_{{\rm i}}}\right)T_h\right\|\ }}\le\]
\[\le \sum^n_{i=1}{\left|{\alpha }_i\right|{\mathop{\overline{\lim} }_{h\to 0} \left\|T_hQ_h\left({\chi }_{{\Delta }_{{\rm i}}}\right)-Q_h\left({\chi }_{{\Delta }_{{\rm i}}}\right)T_h\right\|\ }}=0.\]

\noindent Let $f\in B_0(X)$. Then there are functions ${(f_n)}_{n\in {\mathbb N}}\subset S_0(X)$ such that $f\to f_n$. By preceding relation, we have
\[{\mathop{\lim }_{h\to 0} \left\|T_hQ_h\left(f\right)-Q_h\left(f\right)T_h\right\|\ }=0,\ \forall f\in B_0\left(X\right).\] 

\end{proof}

\begin{proposition} Two asymptotic spectral measures having compact support 

\noindent ${\left\{A_h\right\},\ \left\{B_h\right\}}_{h\in \left.(0,1\right]}:$ ${\Sigma }_X\to B(H)$ are asymptotically commutative, 

\noindent  i.e. ${\mathop{lim}_{h\to 0} \left\|A_h\left({\Delta }_1\right)B_h({\Delta }_2)-B_h({\Delta }_2)A_h({\Delta }_1)\right\|\ }=0$, $\forall {\Delta }_1,{\Delta }_2\in {\Sigma }_{{\rm X}}$, if and only if the corresponding positive asymptotic morphisms ${\left\{Q_h\right\},\left\{P_h\right\}}_{h\in \left.(0,1\right]}:B_0(X)\ \to B(H)$ are asymptotically commutative, 

\noindent  i.e. ${\mathop{lim}_{h\to 0} \left\|Q_h\left(f\right)P_h\left(g\right)-P_h\left(g\right)Q_h\left(f\right)\right\|\ }=0$, $\forall f,g\in B_0\left(X\right)$.

\label{prop 5.2}
\end{proposition} 

\begin{proof} Since $A_h\left({\Delta }_1\right)=Q_h({\chi }_{{\Delta }_{{\rm 1}}})\in B(H)$ and taking into account Proposition \ref{prop 5.1} iii), we have

\noindent 

\[{\mathop{\lim }_{h\to 0} \left\|Q_h\left({\chi }_{{\Delta }_{{\rm 1}}}\right)P_h\left(g\right)-P_h\left(g\right)Q_h\left({\chi }_{{\Delta }_{{\rm 1}}}\right)\right\|\ }=0,\] 

\noindent $\forall {\Delta }_1\in {\Sigma }_{{\rm X}}$ and $\forall g\in B_0\left(X\right).$

\noindent Let $f\in S_0(X)$. Then there are disjoint sets ($\Delta $\textit{${}_{i}$})\textit{${}_{i=1,n}$} such that $f=\sum^n_{i=1}{{\alpha }_i{\chi }_{{\Delta }_i}}$. By above relation, we have

\[{\mathop{\overline{\lim }}_{h\to 0} \left\|Q_h\left(f\right)P_h\left(g\right)-{P_h\left(g\right)Q}_h\left(f\right)\right\|\ }=\]
\[={\mathop{\overline{\lim }}_{h\to 0} \left\|\sum^n_{i=1}{{\alpha }_i\left(Q_h\left({\chi }_{{\Delta }_{{\rm i}}}\right)P_h\left(g\right)-{P_h\left(g\right)Q}_h\left({\chi }_{{\Delta }_{{\rm i}}}\right)\right)}\right\|\ }\le\]
\[\le \sum^n_{i=1}{{\mathop{\overline{\lim }}_{h\to 0} \left|{\alpha }_i\right|\left\|Q_h\left({\chi }_{{\Delta }_{{\rm i}}}\right)P_h\left(g\right)-P_h\left(g\right)Q_h\left({\chi }_{{\Delta }_{{\rm i}}}\right)\right\|\ }}=\]
\[=\sum^n_{i=1}{\left|{\alpha }_i\right|{\mathop{\lim }_{h\to 0} \left\|Q_h\left({\chi }_{{\Delta }_{{\rm i}}}\right)P_h\left(g\right)-P_h\left(g\right)Q_h\left({\chi }_{{\Delta }_{{\rm i}}}\right)\right\|\ }}=0.\]

\noindent Let $f\in B_0(X)$. Then there are functions ${(f_n)}_{n\in {\mathbb N}}\subset S_0(X)$ such that $f\to f_n$. By preceding relation, we have

\[{\mathop{\lim }_{h\to 0} \left\|Q_h\left(f\right)P_h\left(g\right)-P_h\left(g\right)Q_h\left(f\right)\right\|\ }=0,\forall f,g\in B_b\left(X\right).\] 

\noindent Reciprocal. Since

\[{\mathop{\lim }_{h\to 0} \left\|Q_h\left(f\right)P_h\left(g\right)-{P_h\left(g\right)Q}_h\left(f\right)\right\|\ }=0,\ f,g\in S_0\left(X\right)\subset B_0\left(X\right), \]

\noindent taking $f={{\chi }_{\Delta }}_1$ and $g={{\chi }_{\Delta }}_2$ follows

\[{\mathop{\lim }_{h\to 0} \left\|A_h\left({\Delta }_1\right)B_h({\Delta }_2)-B_h({\Delta }_2)A_h({\Delta }_1)\right\|\ }=0,\ \forall {\Delta }_1,{\Delta }_2\in {\Sigma }_{{\rm X}}.\] 
\end{proof}

\begin{theorem} Two asymptotic spectral measures having compact support are asymptotically equivalent if and only if the corresponding positive asymptotic morphisms are asymptotically equivalent. 

\label{theorem 5.3}
\end{theorem} 

\begin{proof} Let ${\left\{A_h\right\},\ \left\{B_h\right\}}_{h\in \left.(0,1\right]}:{\Sigma }_{{\rm X}}\to B(H)$ be two asymptotic spectral measures having compact support and  ${\left\{Q_h\right\},\left\{P_h\right\}\ }_{h\in \left.(0,1\right]}:B_0(X)\ \to B(H)$ the corresponding positive asymptotic morphisms. We suppose that  ${\left\{A_h\right\},\ \left\{B_h\right\}}_{h\in \left.(0,1\right]}:{\Sigma }_{{\rm X}}\to B(H)$ are asymptotically equivalent, i.e.

\[{\mathop{\lim }_{h\to 0} \left\|A_h\left(\Delta \right)-B_h(\Delta )\right\|\ }=0, \forall \Delta \in {\Sigma }_{{\rm X}}.\] 

\noindent Since 

\[A_h\left(\triangle \right)=Q_h({\chi }_{\Delta }), B_h\left(\triangle \right)=P_h({\chi }_{\Delta }),\]

\noindent from preceding relation, we obtain

\[{\mathop{\lim }_{h\to 0} \left\|Q_h\left({\chi }_{\Delta }\right)-P_h\left({\chi }_{\Delta }\right)\right\|\ }=0,  \forall \Delta \in {\Sigma }_{{\rm X}}.\]

\noindent Let $f\in S_0(X)$. Then there are disjoint sets ($\Delta $\textit{${}_{i}$})\textit{${}_{i=1,n}$} such that $f=\sum^n_{i=1}{{\alpha }_i{\chi }_{{\Delta }_i}}$. By above relation, we have

\[{\mathop{\overline{\lim }}_{h\to o} \left\|Q_h\left(f\right)-P_h(f)\right\|\ }={\mathop{\overline{\lim }}_{h\to o} \left\|Q_h\left(\sum^n_{i=1}{{\alpha }_i{\chi }_{{\Delta }_i}}\right)-P_h\left(\sum^n_{i=1}{{\alpha }_i{\chi }_{{\Delta }_i}}\right)\right\|\ }=\]
\[={\mathop{\overline{\lim }}_{h\to o} \left\|\left(\sum^n_{i=1}{{\alpha }_iQ_h({\chi }_{{\Delta }_i})}\right)-\left(\sum^n_{i=1}{{\alpha }_i{P_h(\chi }_{{\Delta }_i})}\right)\right\|\ }=\]
\[={\mathop{\overline{\lim }}_{h\to o} \left\|\sum^n_{i=1}{{\alpha }_i(Q_h({\chi }_{{\Delta }_i})}{{-P}_h(\chi }_{{\Delta }_i}))\right\|\ }\le \sum^n_{i=1}{\left|{\alpha }_i\right|{\mathop{\lim }_{h\to o} \left\|{{Q_h({\chi }_{{\Delta }_i})-P}_h(\chi }_{{\Delta }_i})\right\|\ }}=0.\] 

\noindent therefore

\[{\mathop{\lim }_{h\to o} \left\|Q_h\left(f\right)-P_h\left(f\right)\right\|\ }=0,\forall f\in S_0\left(X\right).\] 

\noindent Let $f\in B_0(X)$. Then there are functions ${(f_n)}_{n\in {\mathbb N}}\subset S_0(X)$ such that $f\to f_n$. Since $Q_h,P_h\in B(H)$ and from above relation, we have

\[{\mathop{\lim }_{h\to o} \left\|Q_h\left(f\right)-P_h\left(f\right)\right\|\ }=0,\ \forall f\in B_0\left(X\right).\]

\noindent Reciprocal. We suppose that ${\left\{Q_h\right\},\left\{P_h\right\}\ }_{h\in \left.(0,1\right]}:B_0(X)\ \to B(H)$ are asymptotically equivalent, i.e. ${\mathop{\lim }_{h\to o} \left\|Q_h\left(f\right)-P_h\left(f\right)\right\|\ }=0$, $\forall f\in B_0\left(X\right).$

\noindent Taking $f={\chi }_{\Delta }$ we have

\[{\mathop{\lim }_{h\to 0} \left\|A_h\left(\Delta \right)-B_h(\Delta )\right\|\ }={\mathop{\lim }_{h\to 0} \left\|Q_h\left({\chi }_{\Delta }\right)-P_h({\chi }_{\Delta })\right\|\ }=0, \forall \Delta \in {\Sigma }_{{\rm X}}.\] 

\end{proof}

\begin{proposition} Let ${\left\{A_h\right\}}_{h\in \left.(0,1\right]}:{\Sigma }_X\to B(H)$ be an asymptotic spectral measure having property $A_h(C_X)\ \subset \ { B}$ $\forall h\in \left.(0,1\right]$ and ${\left\{Q_h\right\}}_{h\in \left.(0,1\right]}:B_0(X)\ \to B\left(H\right)$ the corresponding positive asymptotic morphism. Then

\[spec\left(A_h\right)=supp\left(Q_h\right),\forall h\in \left.(0,1\right].\] 

\label{prop 5.4}
\end{proposition}

\begin{proof} Let $\Delta $ be an open set such that $supp(Q_h)\cap \triangle =\emptyset $. Then

\[Q_h\left(f\right)=0,\ \forall f\ with\ supp(f)\subset \triangle .\]

\noindent Taking $f={\chi }_{\Delta }$, we have ${A_h\left(\triangle \right)=Q}_h\left({\chi }_{\Delta }\right)=0$. Thus $\Delta \subset X\backslash spec(A_h)$, for each open set $\Delta $ such that $supp(Q_h)\cap \triangle =\emptyset $. Therefore

\[spec\left(A_h\right)\subseteq supp\left(Q_h\right),\ \forall h\in \left.(0,1\right].\]

\noindent Reciprocal. Let $\Delta $ be an open set such that${\rm \ }\Delta \subset X\backslash spec(A_h)$. Then $Q_h\left({\chi }_{\Delta }\right)=A_h\left(\triangle \right)=0$.

\noindent Let $f\in S_0(X)$ such that $supp\left(f\right)\subset X\backslash spec(A_h)$. Then there are disjoin sets ($\Delta $\textit{${}_{i}$})\textit{${}_{i=1,n}$} such that $f=\sum^n_{i=1}{{\alpha }_i{\chi }_{{\Delta }_i}}$ and  ${\triangle }_i\subset X\backslash spec(A_h),\ \forall i=\overline{1,n}.$ From the preceding relation, we have

\[Q_h\left(f\right)=Q_h\left(\sum^n_{i=1}{{\alpha }_i{\chi }_{{\Delta }_i}}\right)=\sum^n_{i=1}{{\alpha }_iQ_h({\chi }_{{\Delta }_i})}=0.\]

\noindent Let $f\in B_0(X)$ such that $supp\left(f\right)\subset X\backslash spec(A_h)$. Then there are functions ${(f_n)}_{n\in {\mathbb N}}\subset B_0(X)$ such that $f\to f_n$ and $supp\left(f_n\right)\subset X\backslash spec(A_h),\ \forall n\in {\mathbb N}$. From the preceding relation, we have $Q_h\left(f\right)=0$, $\forall f\in B_0(X)$ such that $supp\left(f\right)\subset X\backslash spec(A_h)$. Thus

\[supp\left(Q_h\right)\subseteq spec\left(A_h\right),\ \forall h\in \left.(0,1\right].\] 

\end{proof}

\begin{remark} Let ${\left\{A_h\right\}}_{h\in \left.(0,1\right]}:{\Sigma }_X\to B(H)$ be an asymptotic spectral measure having property $A_h(C_X)\ \subset \ { B}$ $\forall h\in \left.(0,1\right]$ and ${\left\{Q_h\right\}}_{h\in \left.(0,1\right]}:B_0(X)\ \to B\left(H\right)$ the corresponding positive asymptotic morphism. Then $\left\{A_h\right\}$ has compact support if and only if $\left\{Q_h\right\}$ has compact support.
\end{remark}

\begin{proposition} Let ${\left\{A_h\right\}}_{h\in \left.(0,1\right]}:{\Sigma }_X\to B(H)$ be an asymptotic spectral measure having compact support and ${\left\{Q_h\right\}}_{h\in \left.(0,1\right]}:B_0(X)\ \to B\left(H\right)$ the corresponding positive asymptotic morphism. The corresponding positive asymptotic morphism of asymptotic spectral measure ${\ \left\{A^a_h\right\}}_{h\in \left.(0,1\right]}:{\Sigma }_X\to B(H)$, given by $A^a_h\left(b\right)=A_h\left(a\cap b\right)$ $\forall b\in \ B$ and $\forall h\in \left.(0,1\right]$, is ${\ \left\{Q^a_h\right\}}_{h\in \left.(0,1\right]}:B_0\to B(H)$ given by

\[Q^a_h\left(f\right)=Q_h\left({\chi }_af\right),\ \forall f\in B_0\left(X\right){\rm \ and\ }\forall h\in \left.(0,1\right].\] 

\label{prop 5.5}
\end{proposition}

\begin{proof}  Since ${\left\{A_h\right\}}_{h\in \left.(0,1\right]}:{\Sigma }_X\to B\left(H\right)$ is an asymptotic spectral measure having compact support, by Proposition \ref{prop 3.13} follows that ${\ \left\{A^a_h\right\}}_{h\in \left.(0,1\right]}:{\Sigma }_X\to B(H)$, given by relation $A^a_h\left(b\right)=A_h\left(a\cap b\right)$ for  $\forall b\in \ B$ and $\forall h\in \left.(0,1\right]$, is an asymptotic spectral measure. In addition, by Remark of Proposition \ref{prop 3.14}, results that $\ \left\{A^a_h\right\}$ has a compact support.

\noindent Taking into account the following relation ${\chi }_{a\cap b}={\chi }_a{\chi }_b$, we have

\[A^a_h\left(b\right)=A_h\left(a\cap b\right)=Q_h({\chi }_{a\cap b})={Q_h(\chi }_a{\chi }_b)=Q^a_h({\chi }_b),\]

\noindent $\forall b\in {\rm \ }B\ {\rm \ }$ and  $\forall h\in \left.(0,1\right].\ $

\noindent Let$f\in B_0{\rm (}{\mathbb C}{\rm )}$. Since $\left\{Q_h\right\}$ is the corresponding positive asymptotic morphism of asymptotic spectral measure ${\rm \ }\left\{A_h\right\}$ and having in view that

\[A^a_h(x)={\chi }_a(x)A_h(x),\] 

\noindent we have that

\[Q^a_h\left(f\right)=\ \int_X{f\left(x\right){\chi }_a(x)dA_h(x)}=\]
\[= \int_a{f(x)dA_h(x)}=\int_X{f\left(x\right)d{\chi }_a(x)A_h(x)}=\ \int_X{f(x)dA^a_h(x)}.\] 

\end{proof}

\begin{remark} Even if $\left\{Q_h\right\}$ is a positive asymptotic morphism, $\left\{Q^a_h\right\}$ is not necessary a positive asymptotic morphism, because $Q^a_h\left(1\right)=Q_h\left({\chi }_a\right)$, $\forall h\in \left.(0,1\right]$.
\end{remark}

\begin{corollary} Let ${\left\{A_h\right\}}_{h\in \left.(0,1\right]}:{\Sigma }_X\to B\left(H\right)$ and ${\ \left\{A^a_h\right\}}_{h\in \left.(0,1\right]}:{\Sigma }_X\to B(H)$ as in the preceding Proposition. Then

\[{\mathop{\lim }_{h\to 0} \left\|\int_X{f(x)dA^a_h(x)}-A_h(a)\int_X{f(x)dA_h(x)}\right\|\ }=0.\] 

\noindent $\forall f\in B_0\left(X\right)$.
\label{corollary 5.6}
\end{corollary}

\begin{proof} Let  ${\left\{Q_h\right\}}_{h\in \left.(0,1\right]}:B_0(X)\ \to B\left(H\right)$ the corresponding positive asymptotic morphism of $\left\{A_h\right\}$ and $\left\{Q^a_h\right\}$ given by $Q^a_h\left(f\right)=Q_h\left({\chi }_af\right)$, $\forall f\in B_0\left(X\right)$. Then

\[{\mathop{\overline{\lim }}_{h\to 0} \left\|\int_X{f(x)dA^a_h(x)}-A_h(a)\int_X{f(x)dA_h(x)}\right\|\ }={\mathop{\overline{\lim }}_{h\to 0} \left\|Q^a_h(f)-{A_h(a)Q}_h(f)\right\|\ }=\]
\[={\mathop{\overline\lim }_{h\to 0} \left\|Q_h({\chi }_af)-{Q_h({\chi }_a)Q}_h(f)\right\|\ }=0.\] 

\noindent $\forall f\in B_0\left(X\right)$.
\end{proof}

\begin{corollary} Let ${\left\{Q_h\right\},\ \left\{P_h\right\}}_{h\in \left.(0,1\right]}:B_0(X)\ \to B\left(H\right)$ be two positive asymptotic morphisms and ${\ \left\{Q^a_h\right\},\ \ \left\{P^a_h\right\}}_{h\in \left.(0,1\right]}:B_0\to B(H)$as in Proposition 5.5. Then ${\ \left\{Q^a_h\right\},\ \ \left\{P^a_h\right\}}_{h\in \left.(0,1\right]}$ are asymptotically equivalent if and only if ${\ \left\{Q^a_h\right\},\ \left\{P^a_h\right\}}_{h\in \left.(0,1\right]}$ are asymptotically equivalent for any $a$\textit{ $\in $ B .}
\label{corollary 5.7}
\end{corollary}

\begin{proof} Applying Proposition \ref{prop 5.5}  and Proposition \ref{prop 3.15}
\end{proof}

\begin{theorem} Let ${\left\{A_h\right\}}_{h\in \left.(0,1\right]}:{\Sigma }_X\to B(H)$ be an asymptotic spectral measure having property $A_h(C_X)\ \subset \ {\mathcal B}\ \forall h\in \left.(0,1\right]$ and ${\left\{Q_h\right\}}_{h\in \left.(0,1\right]}:B_0(X)\ \to B\left(H\right)$ the corresponding positive asymptotic morphism. Then

\[spec\left(\left\{A_h\right\}\right)=supp(\left\{Q_h\right\}).\] 

\label{theorem 5.8}
\end{theorem}

\begin{proof} Let $\Delta $ be an open set such that $supp(\left\{Q_h\right\})\cap \triangle =\emptyset $. Then

\[{\mathop{\lim }_{h\to 0} \left\|Q_h(f)\right\|\ }=0,\ \forall f\ with\ supp(f)\subset \triangle .\]

\noindent Taking $={\chi }_{\Delta }$, we have

\[{\mathop{\lim }_{h\to 0} \left\|A_h\left(\triangle \right)\right\|\ }={\mathop{\lim }_{h\to 0} \left\|Q_h({\chi }_{\Delta })\right\|\ }=0.\]

\noindent Thus  $\Delta \subset X\backslash spec(\left\{A_h\right\})$, for each open set $\Delta $ such that $supp(\left\{Q_h\right\})\cap \triangle =\emptyset $. Therefore

\[spec\left(\left\{A_h\right\}\right)\subseteq supp(\left\{Q_h\right\}).\]

\noindent Reciprocal. Let $\Delta $ be an open set such that${\rm \ }\Delta \subset X\backslash spec(A_h)$. Then

\[Q_h\left({\chi }_{\Delta }\right)=A_h\left(\triangle \right)=0.\]

\noindent Let $f\in S_0(X)$ such that $supp\left(f\right)\subset X\backslash spec(\left\{A_h\right\})$. Then there are disjoin sets ($\Delta $\textit{${}_{i}$})\textit{${}_{i=1,n}$} such that $f=\sum^n_{i=1}{{\alpha }_i{\chi }_{{\Delta }_i}}$ and ${\triangle }_i\subset X\backslash spec(\left\{A_h\right\}),\ \forall i=\overline{1,n}.$ From the previous relation, we have

\[Q_h\left(f\right)=Q_h\left(\sum^n_{i=1}{{\alpha }_i{\chi }_{{\Delta }_i}}\right)=\sum^n_{i=1}{{\alpha }_iQ_h({\chi }_{{\Delta }_i})}=0.\]

\noindent When $h\to 0$ into the prior relation, it results

\[{\mathop{\lim }_{h\to 0} \left\|Q_h(f)\right\|\ }=0,\]

\noindent $\forall f\in S_0\left(X\right)$ such that $supp\left(f\right)\subset X\backslash spec(\left\{A_h\right\})$.

\noindent Let $f\in B_0(X)$ such that $supp\left(f\right)\subset X\backslash spec(A_h)$. Then there are functions ${(f_n)}_{n\in {\mathbb N}}\subset S_0(X)$ such that $f\to f_n$ and $supp\left(f_n\right)\subset X\backslash spec(A_h),\ \forall n\in {\mathbb N}$. From the previous relation, we have

\[{\mathop{\lim }_{h\to 0} \left\|Q_h(f)\right\|\ }=0,\]

\noindent $\forall f\in B_0(X)$ such that $supp\left(f\right)\subset X\backslash spec(A_h)$. 

\noindent Therefore

\[supp(\left\{Q_h\right\})\subseteq spec\left(\left\{A_h\right\}\right).\] 

\end{proof}

\begin{corollary} Let ${\left\{Q_h\right\},\ \left\{P_h\right\}}_{h\in \left.(0,1\right]}:B_0(X)\ \to B\left(H\right)$ be two corresponding positive asymptotic morphisms having property $Q_h(C_0(X))\ \subset \ {\mathcal B}\ $ and  $P_h(C_0(X))\ \subset \ {\mathcal B}$ $\forall h\in \left.(0,1\right]$. Then
 
\[supp(\left\{Q_h\right\})=\ supp(\left\{P_h\right\}).\] 

\label{corollary 5.9}
\end{corollary}

\begin{proof} By Theorem \ref{theorem 5.8} and Theorem \ref{theorem 3.12}.
\end{proof} 

\begin{corollary} Let ${\left\{Q_h\right\},\ \left\{P_h\right\}}_{h\in \left.(0,1\right]}:B_0(X)\ \to B\left(H\right)$ be two corresponding positive asymptotic morphisms of two compact asymptotic spectral measures. Then $\left\{Q_h\right\}$ has compact support if and only if $\left\{P_h\right\}$ has compact support.
\label{corollary 5.10}
\end{corollary}

\begin{proof} By Theorem \ref{theorem 5.8} and Theorem \ref{theorem 3.12}.
\end{proof}   

\begin{theorem} Let ${\left\{A_h\right\}}_{h\in \left.(0,1\right]}:{\Sigma }_X\to B(H)$ be an asymptotic spectral measure having compact support and  ${\ \left\{A^a_h\right\}}_{h\in \left.(0,1\right]}:{\Sigma }_X\to B(H)$, given by $A^a_h\left(b\right)=A_h\left(a\cap b\right)$ $\forall b\in \ B$ and $\forall h\in \left.(0,1\right]$. Then

\[spec\left(\left\{A^a_h\right\}\right)=\overline{a}\cap spec\left(\left\{A_h\right\}\right),\ \forall h\in \left.(0,1\right].\] 

\label{theorem 5.11}
\end{theorem}

\begin{proof} We show that

\[spec\left(\left\{A^a_h\right\}\right)\subseteq \overline{a}\cap spec\left(\left\{A_h\right\}\right).\]

\noindent By Proposition \ref{prop 3.14} we have

\[spec(A^a_h)\subset \overline{a}\cap spec(A_h).\]

\noindent By Remark \ref{rem 3.9} i), it follows

\[spec(\left\{A^a_h\right\})\subset \overline{a}.\]

\noindent Let\textit{ b} be a compact set such that $b\subset {\mathbb C}\backslash spec(\left\{A_h\right\})$. Then there are open sets ${(b_i)}_{i=\overline{1,n}\ }$ such that $b\subset \bigcup^n_{i=1}{b_i,\ \ b_i\subset {\mathbb C}\backslash spec(\left\{A_h\right\})}$ and it results

\[{\mathop{\lim }_{h\to 0} \left\|A_h\left(b_i\right)\right\|\ }=0.\]

\noindent  Since each $A_h$ is additive, we have

\[A_h\left(b\right)\le A_h\left(\bigcup^n_{i=1}{b_n}\right)=\sum^n_{i=1}{A_h\left(b_i\right)}.\]

\noindent When $h\to 0$, it follows

\[{\mathop{\lim }_{h\to 0} \left\|A_h\left(b\right)\right\|\ }\le \sum^n_{i=1}{{\mathop{\lim }_{h\to 0} \left\|A_h\left(b_i\right)\right\|\ }}=0.\] 

\noindent How

\[{\mathop{\lim }_{h\to 0} \left\|A^a_h\left(b\right)\right\|\ }={\mathop{\lim }_{h\to 0} \left\|A_h\left(a\cap b\right)\right\|\ }\le\]
\[\le \sum^n_{i=1}{{\mathop{\lim }_{h\to 0} \left\|A_h\left(a\cap b_i\right)\right\|\ }}=\sum^n_{i=1}{{\mathop{\lim }_{h\to 0} \left\|A_h\left(a)A_h(b_i\right)\right\|\ }}\le\]
\[\le \sum^n_{i=1}{{\mathop{\lim }_{h\to 0} \left\|A_h\left(b_i\right)\right\|\ }}=0, \]

\noindent it results $b\subset {\mathbb C}\backslash spec\left(\left\{A^a_h\right\}\right)\ $for each compact set \textit{b} such that $b\subset {\mathbb C}\backslash spec(\left\{A_h\right\})$. Therefore 

\[spec(\left\{A^a_h\right\})\subseteq spec\left(\left\{A_h\right\}\right)\]

\noindent (by regularity of measures $A_h$).

\noindent Reciprocal. Let ${\left\{Q_h\right\}}_{h\in \left.(0,1\right]}:B_0(X)\ \to B\left(H\right)$ be the corresponding positive asymptotic morphism of $\left\{A_h\right\}$. We show 

\[supp\left(\left\{Q^a_h\right\}\right)\supseteq \overline{a}\cap supp\left(\left\{Q_h\right\}\right),\]

\noindent where ${\left\{Q^a_h\right\}}_{h\in \left.(0,1\right]}:B_0\to B(H)$ is given by

\[Q^a_h\left(f\right)=Q_h\left({\chi }_af\right).\]

\noindent Let $a$ be a set such that  $\overline{a}\cap supp\left(\left\{Q_h\right\}\right)\subset F$. Let $f\in B_0\left(X\right)$ such that $supp\left(f\right)\subset X\backslash F$. Then $supp\left(f\right)\bigcap F=\emptyset $ and thus ${\mathop{\lim }_{h\to 0} \left\|Q_h\left(f\right)\right\|\ }=0.\ $

\noindent Let ${{\rm \ }\left\{Q^a_h\right\}}_{h\in \left.(0,1\right]}:B_0\to B(H)$ given by

\noindent 

\noindent $Q^a_h\left(f\right)=Q_h({\chi }_af)$, $\forall f\in B_0(X)$ and $\forall h\in \left.(0,1\right]$. 

\noindent 

\noindent The, form the previous relation and since ${\mathop{\lim }_{h\to 0} \left\|Q_h\left({\chi }_a\right)\right\|\ }\le 1$, we have

\[{\mathop{\lim }_{h\to 0} \left\|Q^a_h\left(f\right)\right\|\ }={\mathop{\lim }_{h\to 0} \left\|Q_h\left({\chi }_af\right)\right\|\ }\le {\mathop{\lim }_{h\to 0} \left\|Q_h\left({\chi }_a\right)Q_h\left(f\right)\right\|\ }\le\]
\[\le {\mathop{\lim }_{h\to 0} \left\|Q_h\left({\chi }_a\right)\right\|\ }{\mathop{\lim }_{h\to 0} \left\|Q_h\left(f\right)\right\|\ }\le {\mathop{\lim }_{h\to 0} \left\|Q_h\left(f\right)\right\|\ }=0,\]

\noindent $f\forall \in B_0\left(X\right)$ such that $supp\left(f\right)\subset X\backslash F$. Thus cã $supp\left(\left\{Q^a_h\right\}\right)\subset F$. 

\noindent By Theorem \ref{theorem 5.8}  we have

\[spec\left(\left\{A_h\right\}\right)=supp(\left\{Q_h\right\}),\] 

\noindent thus

\[\overline{a}\cap spec(\left\{A_h\right\})\subseteq spec\left(\left\{A^a_h\right\}\right).\] 

\end{proof}

\noindent Bibliography

\noindent 

\noindent 1. A. Connes and N. Higson, \textit{Deformations, morphismes asymtptiques et K-theorie bivariante}, C. R. Acad. Sci. Paris, Ser. 1 311 (1990), pp. 101--106.

\noindent 2. Dadarlat, M., A Note on Asymptotic Homomorphisms. K-Theory, 8(5); 465-482, 1994.

\noindent 3. D. Martinez and J. Trout,\textbf{ }\textit{Asymptotic spectral measures, quantum mechanics, and E -- theory}, Comm. Math. Phys, vol. 226, pp. 41--60, 2002.

\noindent 4. J. Trout,\textbf{ }\textit{Asymptotic spectral measures: Between Quantum Theory and E-theory}, Comm. Math. Phys.

\noindent 5. S.K. Berberian, \textit{Notes on Spectral Theory}. Princeton: Van Nostrand, 1966.

\noindent 

\noindent 
\\
\\

\end{document}